\documentclass{amsart}

\makeatletter
\@namedef{subjclassname@2020}{%
  \textup{2020} Mathematics Subject Classification}
 
\makeatother 

\usepackage{amsthm,amssymb,amsfonts,latexsym,mathtools,thmtools,amsmath,lscape}
\usepackage[T1]{fontenc}
\usepackage{tikz-cd} 
\usepackage{enumitem} 
\usepackage{longtable}
\usepackage{multirow}
\usepackage{mathrsfs} 
\usepackage{hyperref} 
\usepackage{hyperref}
\usepackage[all]{xy}
\usepackage{booktabs}
\hypersetup{
    colorlinks=true,
    linkcolor=blue,
    filecolor=blue,      
    urlcolor=blue,
    linktocpage=true
}

\newtheorem{theorem}{Theorem}[section]
\newtheorem{lemma}[theorem]{Lemma}
\newtheorem{proposition}[theorem]{Proposition}

\theoremstyle{definition}
\newtheorem{definition}[theorem]{Definition}

\newtheorem{remark}[theorem]{Remark}

\newtheorem{example}[theorem]{Example}

\theoremstyle{remark}

\numberwithin{equation}{section}

\begin{document}

\title[A view toward the smooth geometry of Sklyanin algebras]{A view toward the smooth geometry of Sklyanin algebras}


\author{Karol Herrera}
\address{Universidad ECCI}
\curraddr{Campus Universitario}
\email{kherrerac@ecci.edu.co}
\thanks{}


\author{Sebastián Higuera}
\address{Universidad ECCI}
\curraddr{Campus Universitario}
\email{shiguerar@ecci.edu.co}
\thanks{}


\author{Andr\'es Rubiano}
\address{Universidad ECCI}
\curraddr{Campus Universitario}
\email{arubianos@ecci.edu.co}
\thanks{}



\date{}

\dedicatory{Dedicated to Ren\'e Veloza}

\begin{abstract} 

We study the differential smoothness of Sklyanin algebras in three and four variables. We show that all non-degenerate three-dimensional cases are differentially smooth, while none of the four-dimensional Sklyanin algebras admit a connected integrable differential calculus of suitable dimension.

\end{abstract}

\maketitle


\section{Introduction}

The Sklyanin algebras originated in the context of integrable systems and the quantum inverse scattering method. In particular, they were introduced by Sklyanin in the early 1980s as a family of noncommutative algebras associated with solutions to the quantum Yang–Baxter equation \cite{Sklyanin1982}. These algebras provide a deformation of polynomial rings and are closely related to the geometry of elliptic curves and their automorphisms.

The three-dimensional Sklyanin algebra, denoted typically by $S(p,q,r)$, is generated by three variables subject to three quadratic relations depending on complex parameters $(p,q,r)$ satisfying a non-degeneracy condition. These algebras are regular in the sense of Artin and Schelter and serve as fundamental examples in the theory of noncommutative projective geometry.

In the seminal work by Artin, Tate, and Van den Bergh \cite{ArtinTateVandenBergh1990}, it was shown that these algebras define noncommutative analogues of projective planes and are intimately connected with elliptic curves and the corresponding point schemes. The three-dimensional Sklyanin algebras are Artin–Schelter (AS) regular of global dimension 3 and exhibit rich geometric and homological properties.

The four-dimensional Sklyanin algebras were later studied and classified in the work of Smith and Stafford \cite{SmithStafford1992} and are generated by four variables with six homogeneous quadratic relations. Similar to the three-dimensional case, they are connected to elliptic curves and automorphisms, though the geometry becomes more subtle, involving degeneracies and rational degenerations \cite{LevasseurStafford1991}.

Both the three- and four-dimensional Sklyanin algebras have been extensively studied as examples of noncommutative Calabi–Yau algebras, and they play a prominent role in the understanding of Gelfand–Kirillov dimension, deformation theory, and noncommutative smooth geometry.

In \cite{BrzezinskiElKaoutitLomp2010}, Brzezi{\'n}ski and collaborators introduced a method to construct {\em differential calculi} that support the existence of hom-connections. Their framework relies on the concept of {\em twisted multi-derivations}, leading to a first-order differential calculus $\Omega^1(A)$ that is free as both a left and right $A$-module. The module $\Omega^1(A)$ can be interpreted as the analogue of the module of sections of the cotangent bundle on a manifold encoded by the algebra $A$, thereby modeling situations akin to parallelizable manifolds or coordinate patches with trivial bundles. Subsequently, in \cite{Brzezinski2011}, Brzezi{\'n}ski suggested that \textquotedblleft one should expect $\Omega^1(A)$ to be a finitely generated and projective module over $A$ (thus corresponding to sections of a non-trivial vector bundle by the Serre--Swan theorem)\textquotedblright\ \cite[p.~885]{Brzezinski2011}. There, he broadened the earlier construction to include the case where $\Omega^1(A)$ is finitely generated and projective.

In their work \cite{BrzezinskiSitarz2017}, Brzezi{\'n}ski and Sitarz introduced a refined concept of algebraic smoothness, which they called {\em differential smoothness}. This approach is based on the use of differential graded algebras of a given dimension that possess a noncommutative counterpart of the classical Hodge star isomorphism. Specifically, it entails the existence of a top-degree differential form and a version of Poincar\'e duality expressed as an isomorphism between complexes of differential and integral forms. Unlike the previously discussed notion of homological smoothness, this new definition is more constructive and geometric in nature. As explained in \cite[p.~413]{BrzezinskiSitarz2017}, the motivation behind {\em differential smoothness} stems from classical geometry, where every smooth orientable manifold supports not only the de Rham complex but also a complex of {\em integral forms} that is isomorphic to it \cite[Section~4.5]{Manin1997}. In this setting, the de Rham differential can be interpreted as a particular case of a left connection, whereas the corresponding boundary operator in the complex of integral forms exemplifies a right connection.

A wide range of algebraic structures have been examined in the literature concerning their differential smoothness properties (see, for instance, \cite{Brzezinski2015, Brzezinski2016, BrzezinskiElKaoutitLomp2010, BrzezinskiLomp2018, BrzezinskiSitarz2017, DuboisViolette1988, DuboisVioletteKernerMadore1990, Karacuha2015, KaracuhaLomp2014, ReyesSarmiento2022}). These include various quantum analogues of classical spaces such as the quantum two- and three-spheres, the quantum disc and plane, as well as the noncommutative torus. Further examples involve coordinate rings of quantum groups like $SU_q(2)$, the noncommutative pillow and cone algebras, quantum polynomial rings, and several types of Ore extensions and skew polynomial algebras in low dimensions. Studies have also focused on Hopf algebra domains of Gelfand–Kirillov dimension two that do not satisfy a polynomial identity (non-PI), as well as on algebras arising from deformations of classical orbifolds including the pillow, lens spaces, and certain singular cones. Interestingly, a number of these algebras also satisfy homological smoothness in the sense of Van den Bergh.


Recent work by Reyes and the third author has concentrated on establishing sufficient (and in some instances, necessary) criteria for the differential smoothness of various families of noncommutative graded algebras. Among these are bi-quadratic algebras in three generators possessing PBW bases \cite{RubianoReyes2024DSBiquadraticAlgebras}, double Ore extensions characterized by the Hilbert series type $(14641)$ \cite{RubianoReyes2024DSDoubleOreExtensions}, and skew PBW extensions over commutative polynomial rings \cite{RubianoReyes2024DSSPBWKt}. In each case, the smoothness property is investigated through explicit algebra presentations, detailed scrutiny of their defining relations, and the construction of derivations or differential operators satisfying suitable lifting conditions. 

Most recently, in \cite{Rubiano2025DSASRegular5D}, the third author established obstructions to differential smoothness for AS-regular algebras in dimension five with fewer generators than their Gelfand–Kirillov dimension, thus connecting algebraic structure with the existence of differential calculi in a precise way.

This line of inquiry has also been extended to encompass broader algebraic families, such as diffusion algebras \cite{RubianoReyes2024DSDifA} and general classes of skew PBW extensions \cite{RubianoReyes2024DSSBWR}. Notably, a comprehensive differential smoothness analysis of PBW bi-quadratic algebras has been recently completed, covering families with more complex relational structures \cite{RubianoReyes2025DSnBQ}. These findings provide significant insight into the interaction between algebraic presentations and homological conditions in the context of noncommutative differential geometry, particularly within the class of Artin–Schelter regular algebras.

Motivated by the growing interest in the differential smoothness of noncommutative graded algebras, our goal in this paper is to explore this property within the family of Sklyanin algebras. These algebras, originally arising in the context of quantum integrable systems, serve as fundamental examples of Artin–Schelter regular algebras and have been widely studied due to their rich algebraic and geometric structures. Despite their well-understood homological behavior and deep connections with elliptic curves, little is known about whether they admit differential calculi satisfying the conditions of differential smoothness in the sense of Brzezi{\'n}ski and Sitarz.

The diversity in the number of generators, the complexity of their defining relations, and the geometric content encoded in their point schemes make these algebras compelling candidates for a systematic investigation of differential smoothness. Moreover, since they exhibit regularity, finite global dimension, and noetherianity, key features that often suggest favorable homological behavior, it is natural to ask whether these properties are sufficient to guarantee the existence of a suitable differential graded structure. This work seeks to answer such questions by applying and adapting recent techniques developed in the study of skew PBW extensions and other noncommutative settings, now within the Sklyanin framework.

The structure of the paper is as follows. In Section~\ref{PreliminariesDifferentialsmoothnessofbi-quadraticalgebras}, we present the foundational concepts related to differential smoothness, along with the essential notation that will be used throughout the text. This section also includes a concise overview of the fundamental properties of Sklyanin algebras with three and four generators, ensuring that the exposition remains self-contained. Section~\ref{DifferentialandintegralcalculusSkly} is devoted to the main results of this work: Theorem~\ref{DS3Skly}, establishing that the non-degenerate three-dimensional Sklyanin algebras are differentially smooth, and Theorem~\ref{NODS4Skly}, which says that no four-dimensional Sklyanin algebra satisfies the conditions for differential smoothness.

Throughout the paper, $\mathbb{N}$ denotes the set of natural numbers including zero. The word ring means an associative ring with identity not necessarily commutative. $Z(R)$ denotes the center of the ring $R$. All vector spaces and algebras (always associative and with unit) are over a fixed field $\Bbbk$. $\Bbbk^{*}$ denotes the non-zero elements of $\Bbbk$. As usual, the symbols $\mathbb{R}$ and $\mathbb{C}$ denote the fields of real and complex numbers, respectively. 

\section{Preliminaries}\label{PreliminariesDifferentialsmoothnessofbi-quadraticalgebras}

We begin by reviewing the essential background on differential smoothness and Sklyanin algebras, which will serve as the foundation for the subsequent developments in this work.

\subsection{Differential smoothness of algebras}\label{DefinitionsandpreliminariesDSA}




We adopt the framework for differential smoothness introduced by Brzezi{\'n}ski and Sitarz in \cite[Section~2]{BrzezinskiSitarz2017}, see also \cite{Brzezinski2008, Brzezinski2014} for further details.

\begin{definition}[{\cite[Section~2.1]{BrzezinskiSitarz2017}}]
Let $A$ be an algebra over a field $\Bbbk$.
\begin{enumerate}
    \item[\rm (i)] A \emph{differential graded algebra} (DGA) is a non-negatively graded algebra $\Omega = \bigoplus_{n \in \mathbb{N}} \Omega^n$ equipped with an associative product $\wedge$ and a degree-one linear map $d: \Omega^\bullet \to \Omega^{\bullet +1}$ such that $d \circ d = 0$ and $d$ satisfies the graded Leibniz rule: for all homogeneous $\omega, \eta \in \Omega$,
    \[
    d(\omega \wedge \eta) = d(\omega) \wedge \eta + (-1)^{\deg(\omega)} \omega \wedge d(\eta).
    \]
    
    \item[\rm (ii)] A DGA $(\Omega, d)$ is called a \emph{differential calculus over an algebra} $A$ if $\Omega^0 = A$ and for all $n \in \mathbb{N}$, the module $\Omega^n$ is generated by $dA$ via wedge products as
    \[
    \Omega^n = A \cdot \underbrace{dA \wedge \dotsb \wedge dA}_{n\text{ times}} = \underbrace{dA \wedge \dotsb \wedge dA}_{n\text{ times}} \cdot A.
    \]
    This property is referred to as the \emph{density condition}. The differential calculus is said to be \emph{connected} if $\ker(d|_{\Omega^0}) = \Bbbk$.
    
    \item[\rm (iii)] A differential calculus $(\Omega, d)$ has \emph{dimension $n$} if $\Omega^n \neq 0$ and $\Omega^m = 0$ for all $m > n$. It is said to \emph{admit a volume form} if $\Omega^n$ is isomorphic to $A$ as both a left and a right $A$-module.
\end{enumerate}
\end{definition}

The condition that $\Omega^n A$ is isomorphic to $A$ as a right $A$-module implies the existence of an element $\omega \in \Omega^n A$ such that every $n$-form can be uniquely expressed as $\omega a$, for $a \in A$. In this case, $\omega$ is said to be a \emph{right $A$-module generator} of $\Omega^n A$. If, additionally, $\omega$ serves as a free generator from the left, i.e., every $n$-form can also be uniquely written as $a \omega$, then $\omega$ is called a \emph{volume form} on $\Omega A$.

Associated to a volume form $\omega$, there exists a canonical right $A$-module isomorphism $\pi_{\omega}: \Omega^n A \to A$, defined by
\begin{equation}\label{BrzezinskiSitarz20172.1}
    \pi_{\omega}(\omega a) = a, \quad \text{for all } a \in A.
\end{equation}

Furthermore, the left $A$-module structure of $\Omega^n A$ induces an algebra endomorphism $\nu_{\omega}: A \to A$ characterized by the relation
\begin{equation}\label{BrzezinskiSitarz20172.2}
    a \omega = \omega \nu_{\omega}(a), \quad \text{for all } a \in A.
\end{equation}
When $\omega$ is a volume form, the map $\nu_{\omega}$ is in fact an algebra automorphism of $A$.

We now recall the basic elements of the \emph{integral calculus} associated to a differential calculus, which can be regarded as its dual. For a detailed discussion, see \cite{Brzezinski2008, BrzezinskiElKaoutitLomp2010}.

Let $(\Omega A, d)$ be a differential calculus over $A$. Since each $\Omega^n A$ is an $A$-bimodule, one can define its right dual as
\[
\mathcal{I}_n A := \operatorname{Hom}_A(\Omega^n A, A),
\]
the space of right $A$-linear maps from $\Omega^n A$ to $A$. Each $\mathcal{I}_n A$ naturally inherits an $A$-bimodule structure given by:
\[
(a \cdot \phi \cdot b)(\omega) := a \, \phi(b\omega), \quad \text{for all } \phi \in \mathcal{I}_n A,\ \omega \in \Omega^n A,\ a, b \in A.
\]

The direct sum $\mathcal{I}A := \bigoplus_{n \geq 0} \mathcal{I}_n A$ becomes a right module over the DGA $\Omega A$, with action defined via the wedge product:
\begin{equation}\label{BrzezinskiSitarz2017}
    (\phi \cdot \omega)(\omega') := \phi(\omega \wedge \omega'), \quad \text{for } \phi \in \mathcal{I}_{n+m} A,\ \omega \in \Omega^n A,\ \omega' \in \Omega^m A.
\end{equation}

\begin{definition}[{\cite[Definition 2.1]{Brzezinski2008}}]
A \emph{divergence} (also known as a \emph{hom-connection}) on $A$ is a linear map $\nabla: \mathcal{I}_1 A \to A$ satisfying the relation
\begin{equation}\label{divergence-rule}
    \nabla(\phi \cdot a) = \nabla(\phi) a + \phi(da), \quad \text{for all } \phi \in \mathcal{I}_1 A,\ a \in A.
\end{equation}
\end{definition}

This operator $\nabla$ can be extended to a sequence of maps $\nabla_n: \mathcal{I}_{n+1} A \to \mathcal{I}_n A$ for each $n \geq 0$, via the formula:
\begin{equation}\label{divergence-extension}
    \nabla_n(\phi)(\omega) = \nabla(\phi \cdot \omega) + (-1)^{n+1} \phi(d\omega),
\end{equation}
for all $\phi \in \mathcal{I}_{n+1} A$ and $\omega \in \Omega^n A$.

Combining equations \eqref{divergence-rule} and \eqref{divergence-extension}, one obtains a graded Leibniz identity for $\nabla_n$:
\begin{equation}
    \nabla_n(\phi \cdot \omega) = \nabla_{m+n}(\phi) \cdot \omega + (-1)^{m+n} \phi \cdot d\omega,
\end{equation}
for all $\phi \in \mathcal{I}_{m+n+1} A$ and $\omega \in \Omega^m A$, as established in \cite[Lemma 3.2]{Brzezinski2008}. In particular, when $n = 0$ and $\mathcal{I}_0 A \cong A$, this reduces to the classical Leibniz rule.

\begin{definition}[{\cite[Definition 3.4]{Brzezinski2008}}]
Let $M$ be a right $A$-module. The composite map 
\[
F := \nabla_0 \circ \nabla_1: \operatorname{Hom}_A(\Omega^2 A, M) \to M
\]
is referred to as the \emph{curvature} of the hom-connection $(M, \nabla_0)$. The connection is said to be \emph{flat} if $F = 0$, which implies that $\nabla_n \circ \nabla_{n+1} = 0$ for all $n \geq 0$.
\end{definition}

The sequence of maps $\nabla_n$ defines a chain complex $(\mathcal{I}_\bullet A, \nabla)$, termed the \emph{complex of integral forms} on $A$. The canonical projection $\Lambda: A \to A / \operatorname{Im}(\nabla)$ is interpreted as an \emph{integral} over $A$.

Given a left $A$-module $X$ and an algebra automorphism $\nu: A \to A$, one defines the \emph{twisted} left $A$-module ${}^{\nu} X$ by modifying the scalar action as $a \cdot x := \nu(a) x$ for all $a \in A$ and $x \in X$.

The following concept formalizes the notion of a Hodge-type duality between differential and integral forms \cite[p.~112]{Brzezinski2015}.

\begin{definition}[{\cite[Definition 2.1]{BrzezinskiSitarz2017}}]
Let $(\Omega A, d)$ be an $n$-dimensional differential calculus. It is called \emph{integrable} if there exists a chain complex of integral forms $(\mathcal{I} A, \nabla)$ and an algebra automorphism $\nu: A \to A$, along with a family of $A$-bimodule isomorphisms
\[
\Theta_k: \Omega^k A \to {}^{\nu} \mathcal{I}_{n - k} A, \quad \text{for all } k = 0, \dotsc, n,
\]
such that the following diagram commutes:
\[
\begin{tikzcd}
A \arrow{r}{d} \arrow{d}{\Theta_0} & \Omega^{1} A \arrow{r}{d} \arrow{d}{\Theta_1} & \dotsb \arrow{r}{d} & \Omega^{n-1} A \arrow{r}{d} \arrow{d}{\Theta_{n-1}} & \Omega^n A \arrow{d}{\Theta_n} \\
{}^{\nu} \mathcal{I}_n A \arrow{r}{\nabla_{n-1}} & {}^{\nu} \mathcal{I}_{n-1} A \arrow{r}{\nabla_{n-2}} & \dotsb \arrow{r}{\nabla_1} & {}^{\nu} \mathcal{I}_1 A \arrow{r}{\nabla} & {}^{\nu} A
\end{tikzcd}
\]
The element $\omega := \Theta_n^{-1}(1) \in \Omega^n A$ is called the \emph{integrating volume form}.
\end{definition}

Several instances of algebras have been shown to support integrable differential calculi. For example, the matrix algebra $M_n(\mathbb{C})$ admits an $n$-dimensional calculus constructed from derivations, as originally developed by Dubois-Violette and collaborators \cite{DuboisViolette1988, DuboisVioletteKernerMadore1990}. Likewise, the quantum group $SU_q(2)$ supports a left-covariant three-dimensional calculus introduced by Woronowicz \cite{Woronowicz1987}, and the same holds for the corresponding restriction to the quantum standard sphere. Further developments and treatments can be found in \cite{BrzezinskiElKaoutitLomp2010}.

Interestingly, the following result provides a characterization of integrable calculi that does not require explicit knowledge of the associated integral forms. Instead, integrability can be verified through the existence of generating elements that permit a left-right decomposition of differential forms.

\begin{proposition}[{\cite[Theorem 2.2]{BrzezinskiSitarz2017}}]\label{integrableequiva}
Let $(\Omega A, d)$ be an $n$-dimensional differential calculus over an algebra $A$. The following statements are equivalent:
\begin{enumerate}
    \item [\rm (1)] The calculus $(\Omega A, d)$ is integrable.

    \item [\rm (2)] There exists an algebra automorphism $\nu$ of $A$ and $A$-bimodule isomorphisms $\Theta_k: \Omega^k A \to {}^{\nu} \mathcal{I}_{n-k} A$ for $k = 0, \dotsc, n$, such that for all $\omega' \in \Omega^k A$ and $\omega'' \in \Omega^m A$:
    \[
    \Theta_{k+m}(\omega' \wedge \omega'') = (-1)^{(n-1)m} \Theta_k(\omega') \cdot \omega''.
    \]

    \item [\rm (3)] There exists an automorphism $\nu$ of $A$ and an $A$-bimodule map $\vartheta: \Omega^n A \to {}^{\nu} A$ such that all left multiplication maps
    \[
    \ell^{k}_{\vartheta}: \Omega^k A \to \mathcal{I}_{n-k} A, \quad \omega' \mapsto \vartheta \cdot \omega',
    \]
    are bijections for each $k = 0, \dotsc, n$, where the product is as in \eqref{BrzezinskiSitarz2017}.

    \item [\rm (4)] The calculus $(\Omega A, d)$ possesses a volume form $\omega \in \Omega^n A$ such that the maps
    \[
    \ell^k_{\pi_{\omega}}: \Omega^k A \to \mathcal{I}_{n-k} A, \quad \omega' \mapsto \pi_{\omega} \cdot \omega',
    \]
    are bijective for all $k = 0, \dotsc, n-1$, with $\pi_{\omega}$ defined in \eqref{BrzezinskiSitarz20172.1}.
\end{enumerate}
\end{proposition}

In light of Proposition~\ref{integrableequiva}, a volume form $\omega \in \Omega^n A$ qualifies as an \emph{integrating form} precisely when it satisfies condition (4); see also \cite[Remark 2.3]{BrzezinskiSitarz2017}.

The most interesting cases of differential calculi are those where $\Omega^k A$ are finitely generated and projective right or left (or both) $A$-modules \cite{Brzezinski2011}.

\begin{proposition}\label{BrzezinskiSitarz2017Lemmas2.6and2.7}
\begin{enumerate}
\item [\rm (1)] {\rm \cite[Lemma 2.6]{BrzezinskiSitarz2017}} Let $(\Omega A, d)$ be an $n$-dimensional integrable differential calculus over $A$, and suppose it admits an integrating form $\omega$. Then the module $\Omega^k A$ is finitely generated and projective as a right $A$-module provided there exist finitely many elements $\omega_i \in \Omega^k A$ and $\overline{\omega}_i \in \Omega^{n-k} A$ such that for all $\omega' \in \Omega^k A$, we have:
\[
\omega' = \sum_i \omega_i\, \pi_\omega(\overline{\omega}_i \wedge \omega').
\]
    
\item [\rm (2)] {\rm \cite[Lemma 2.7]{BrzezinskiSitarz2017}} Suppose $(\Omega A, d)$ is an $n$-dimensional differential calculus over $A$ possessing a volume form $\omega$. If, for each $k = 1, \ldots, n-1$, there exist finitely many elements $\omega_i^k, \overline{\omega}_i^k \in \Omega^k A$ such that for all $\omega' \in \Omega^k A$:
\[
\omega' = \sum_i \omega_i^k\, \pi_\omega(\overline{\omega}_i^{n-k} \wedge \omega') = \sum_i \nu_\omega^{-1}(\pi_\omega(\omega' \wedge \omega_i^{n-k}))\, \overline{\omega}_i^k,
\]
then $\omega$ serves as an integrating form and each $\Omega^k A$ is finitely generated and projective as both a left and a right $A$-module. Here, $\pi_\omega$ and $\nu_\omega$ are defined by equations \eqref{BrzezinskiSitarz20172.1} and \eqref{BrzezinskiSitarz20172.2}, respectively.
\end{enumerate}
\end{proposition}

As noted by Brzeziński and Sitarz \cite[p.~421]{BrzezinskiSitarz2017}, to effectively connect the integrability of a differential calculus $(\Omega A, d)$ with the structure of the algebra $A$, it is essential to compare the dimension of the calculus with that of the algebra itself. Since many of the algebras under consideration are noncommutative deformations of coordinate rings of affine varieties, the most appropriate dimension concept is the {\em Gelfand–Kirillov dimension}, introduced in \cite{GelfandKirillov1966, GelfandKirillov1966b}. Specifically, for an affine $\Bbbk$-algebra $A$, the {\em Gelfand–Kirillov dimension} of $A$, denoted ${\rm GKdim}(A)$, is defined by:
\[
{\rm GKdim}(A) := \limsup_{n \to \infty} \frac{\log(\dim V^n)}{\log n},
\]
where $V$ is any finite-dimensional $\Bbbk$-subspace of $A$ generating $A$ as an algebra. This value is independent of the choice of $V$. In the case where $A$ is not affine, its Gelfand–Kirillov dimension is defined as the supremum of the dimensions of all its affine subalgebras.

An affine domain with GK-dimension zero is precisely a finite-dimensional division algebra over its center. If ${\rm GKdim}(A) = 1$, then $A$ is a finite module over its center and hence satisfies a polynomial identity. Thus, the Gelfand–Kirillov dimension serves as a noncommutative analog of Krull dimension, measuring how far $A$ is from being finite-dimensional. A thorough treatment of this topic can be found in \cite{KrauseLenagan2000}.

We now present the central concept underpinning this article:

\begin{definition}[{\cite[Definition 2.4]{BrzezinskiSitarz2017}}]\label{BrzezinskiSitarz2017Definition2.4}
Let $A$ be an affine algebra with integer Gelfand–Kirillov dimension $n$. Then $A$ is said to be {\em differentially smooth} if it admits a connected, integrable differential calculus $(\Omega A, d)$ of dimension $n$.
\end{definition}

As a consequence of Definition~\ref{BrzezinskiSitarz2017Definition2.4}, a differentially smooth algebra possesses a well-structured differential framework along with a rigorously defined notion of integration \cite[p.~2414]{BrzezinskiLomp2018}.

\begin{example}
As discussed in the Introduction, numerous instances of noncommutative algebras have been shown to satisfy the criteria for differential smoothness (see, for example, \cite{Brzezinski2015, BrzezinskiElKaoutitLomp2010, BrzezinskiLomp2018, BrzezinskiSitarz2017, Karacuha2015, KaracuhaLomp2014, ReyesSarmiento2022}). A classical example is the polynomial ring $\Bbbk[x_1, \dotsc, x_n]$, which has Gelfand–Kirillov dimension $n$, and its associated exterior algebra provides an $n$-dimensional integrable calculus. Hence, $\Bbbk[x_1, \dotsc, x_n]$ is differentially smooth. Furthermore, it was established in \cite{BrzezinskiElKaoutitLomp2010} that several quantum coordinate algebras, including those of the quantum group $SU_q(2)$, the standard Podleś sphere, and the quantum Manin plane, also possess differential smoothness.
\end{example}

\subsection{Three-dimensional Sklyanin algebras}\label{threeSkly}

Artin and Schelter classified the AS-regular algebras of global dimension $3$, generated in grade $1$ into algebras of type $A$, $B$, $H$, $S_1$, $S_1'$, $S_2$. Any member of the \textquotedblleft type $A$\textquotedblright \ family is now known as the { \em three-dimensional Sklyanin algebra}. The study of the Sklyanin algebra allowed a more complete classification of AS-regular algebras of dimension $3$, since for certain values of $p,q,r$ of the relations of the Sklyanin algebra (see Definition \ref{Def.3-dim SA}) it is AS-regular (when at least two of the parameters $p,q, r$ are nonzero and the equation $p^3 = q^3 = r^3$ does not hold); they are now known as {\em non-degenerate Sklyanin algebras}. Also it was interesting to wonder what happens for parameter values where the algebra is not AS-regular. In the case of the Sklyanin family or type $A$, the non-regular members are known as {\em degenerate Sklyanin algebras}. The non-degenerate case shares many properties with the {\em commutative polynomial ring} $\Bbbk[x,y,z]$ while the degenerate satisfies none of these properties. For more details about these algebras, see Walton \cite{Walton2009} and Iyudu and Shkarin \cite{ IyuduShkarin2017}.

\begin{definition}[{\cite[p. 382]{IyuduShkarin2017}}]\label{Def.3-dim SA}
The {\em three-dimensional Sklyanin algebras} over a field $\Bbbk$, denoted by $S(p, q, r)$, are generated by three noncommuting variables $x, y, z$ subject to three relations given by
\begin{equation}\label{rel3Skly}
    pyz + qzy + rx^2 =0, \ pzx + qxz + ry^2 =0, \ pxy + qyx + rz^2 =0, 
\end{equation}
    
for $(p,q,r)\in \mathbb{P}^2(\Bbbk)$.
\end{definition}

Sklyanin algebras can be classified into two different types, the \textit{non-degenerate Sklyanin algebras} and the \textit{degenerate Sklyanin algebras}. Let $\mathcal{D}$ be the subset of the projective plane $\mathbb{P}^2(\Bbbk)$ consisting of the 12 points given by
\[
\mathcal{D}=\{(1,0,0), (0,1,0),(0,0,1)\} \cup \{(p,q,r) \mid p^3=q^3=r^3 \}.
\]

\begin{definition}[{\cite[Definition 1.4]{Walton2009}}] \label{Def.Deg.nonDeg}  
The algebras $S(p,q,r)$ with $(p,q,r)\in \mathcal{D}$ are called the {\em degenerate three-dimensional Sklyanin algebras}. Such algebras are denoted by $S(p,q,r)$ or $S_{{\rm deg}}$ for short. When $(p,q,r)\in \mathbb{P}^2(\Bbbk)\ \backslash\  \mathcal{D}$ are called the {\em non-degenerate three- dimensional Sklyanin algebras}.
\end{definition}

Note that Sklyanin algebras are \textit{quadratic algebras} since their relations are homogeneous of degree two.

Note that multiplying $(p,q,r) \in  \Bbbk^3$ by a non-zero scalar the algebra $S(p,q,r)$ does not change. However, there are non-proportional triples of parameters which lead to isomorphic Sklyanin algebras (like {\em graded algebras}). The next three results present such isomorphism, important for classifying Sklyanin algebras.

\begin{proposition}[{\cite[Lemma 3.1]{IyuduShkarin2017}}]\label{lemma3.1}   
For every $(p,q,r)\in  \Bbbk^3$, the graded algebras $S(p,q,r)$ and $S(q,p, r)$ are isomorphic. 
\end{proposition}

\begin{proposition}[{\cite[Lemma 3.2]{IyuduShkarin2017}}]\label{lemma3.2} 
Assume that $(p,q,r)\in  \Bbbk^3$ and $\theta \in \Bbbk$ is such that $\theta^3=1$ and $\theta\neq 1$. Then the graded algebras $S(p,q,r)$ and $S(p,q,\theta r)$ are isomorphic. 
\end{proposition}

\begin{lemma}[{\cite[Lemma 3.3]{IyuduShkarin2017}}] \label{lemma3.3}
Assume that $(p,q,r)\in \Bbbk^3$ and $\theta \in \Bbbk$ is such that $\theta^3=1$ and $\theta\neq 1$. Then graded algebras $S(p,q,r)$ and $S(p',q', r')$ are isomorphic, where
\begin{equation*}
    p'=\theta^2p+\theta q+r, \hspace{0.5cm} q'=\theta p+\theta^2q+r, \hspace{0.5cm} r'=p+q+r.
\end{equation*}
\end{lemma}

\begin{proposition}[{\cite[Theorem 1.2]{IyuduShkarin2017}}] \label{Theorem1.2}
    The algebra $S(p, q, r)$ is PBW if and only if at least one of the following conditions is satisfied:
    \begin{enumerate}
        \item [\rm (1)] $pr = qr = 0$;
        \item [\rm (2)] $p^3 = q^3 = r^3$;
        \item [\rm (3)] $(p + q)^3 + r^3 = 0$ and the equation $t^2 + t + 1$ is solvable in $\Bbbk$.
    \end{enumerate}
\end{proposition}

Walton \cite{Walton2009} has shown that the degenerate Sklyanin algebras are nothing like the others: if $(p,q,r) \in \mathcal{D}$, then the algebra $S_{{\rm deg}}$  has infinite global dimension, is not Noetherian, has exponential growth, and has zero divisors, none of which happens when $(p,q,r)\in \mathbb{P}^2_\Bbbk\ \backslash\ \mathcal{D}$. However, Smith in \cite{Smith2012} proves a result for this type of Sklyanin algebras, they are monomial algebras, this fact will be important in the next section. 

\begin{proposition}[{\cite[Theorem 2.1]{Smith2012}}] \label{MonomialAlgebra} 
Let $\Bbbk$ be a field having a primitive cube root of unity, which implies that ${\rm char}(\Bbbk) \neq3$. Suppose $(p,q,r)\in \mathcal{D}$. Then
 \begin{equation*}
            S_{\rm deg}\cong\left\{ \begin{array}{lcc}
             A=\Bbbk\{ x,y,z\}\big/\langle x^2,y^2,z^2\rangle \ \text{if}~ p=q,\\
    \\ A'=\Bbbk\{ x,y,z\}\big/\langle xy,yz,zx\rangle \ \text{if} \ p\neq q.
             \end{array}
   \right.
\end{equation*}
\end{proposition}

\subsection{Four-dimensional Sklyanin algebras}\label{fourSkly}

Sklyanin's interest in these algebras arises from his study of Yang–Baxter matrices and the related {\em Quantum Inverse Scattering Method} (called the {\em Quantum Inverse Problem Method} \cite{Sklyanin1982, Sklyanin1983}) as these algebras provide the general solution to this method corresponding to Baxter's simplest examples of Yang-Baxter matrices.

\begin{definition}[{\cite[p. 746]{ChirvasituSmith2023}}] 
The families of graded algebras denoted by $A(\alpha,\beta,\gamma)$ and depending on a parameter $(\alpha,\beta,\gamma)\in\Bbbk^3$ are generated by three non-commutative variables  $x_0,x_1,x_2,x_3$ subject to three relations given by
\begin{align}\label{FDSAdefinition}
    x_0 x_1  - x_1 x_0 = &\  \alpha(x_2x_3 + x_3x_2), & x_0x_1 + x_1x_0 = &\ x_2x_3 - x_3x_2,\\
    x_0 x_2  - x_2 x_0 = &\  \beta(x_3x_1 + x_1x_3), & x_0x_2 + x_2x_0 = &\ x_3x_1 - x_1x_3,\\
    x_0 x_3  - x_3 x_0 = &\  \gamma(x_1x_2 + x_2x_1), & x_0x_3 + x_3x_0 = &\ x_1x_2 - x_2x_1.
\end{align}
\end{definition}

\begin{definition}[{\cite[p. 747]{ChirvasituSmith2023}}] 
Suppose $\alpha+\beta+\gamma+\alpha\beta\gamma=0$. We call $A(\alpha,\beta,\gamma)$ a {\em four-dimensional Sklyanin algebras} and denote it by $S(\alpha,\beta,\gamma)$. If , in addition, $\{\alpha,\beta,\gamma \} \cap \{0, \pm 1 \}=\emptyset $, we call $S(\alpha, \beta, \gamma)$ a \textit{non-degenerate four-dimensional Sklyanin algebra}. If $\{\alpha,\beta,\gamma \} \cap \{0, \pm 1 \}\neq \emptyset $, $S(\alpha, \beta, \gamma)$ is said to be a {\em degenerate four-dimensional Sklyanin algebra}. 
\end{definition}

The main result of Smith and Stafford \cite{SmithStafford1992} is the following fairly complete description of the structure of $S(\alpha,\beta,\gamma)$.

\begin{proposition}[{\cite[Theorem 0.3]{SmithStafford1992}}] \label{SmithStafford1992Theorem0.3}
Assume that $\{\alpha, \beta,\gamma\}$ is not equal to $\{-1,1,\gamma \}$, $\{\alpha,-1,1 \}$ or $\{1,\beta, -1 \}$. Then:
\begin{enumerate}
   \item [{\rm (1)}] $S(\alpha,\beta,\gamma)$ is a regular graded algebra of dimension 4.
    \item [{\rm (2)}] $H_S(t)=\frac{1}{(1-t)^4}$ is the Hilbert series of a commutative polynomial ring in 4 variables.
   \item [{\rm (3)}]   $S(\alpha,\beta,\gamma)$ is a Noetherian domain.
\end{enumerate}
For the other values of $\alpha,\beta, \gamma$, the ring $S(\alpha,\beta,\gamma)$ has many zero-divisors. 
\end{proposition}

For generic values of $\{\alpha,\beta, \gamma\}$, Proposition \ref{SmithStafford1992Theorem0.3}(2) has been proved by Cherednik \cite[Theorem 14]{Cherednik1986} by regarding $S(\alpha,\beta, \gamma)$ as a deformation of the graded analogue of $U(\mathfrak{so}_3)$; the enveloping algebra of the Lie algebra $\mathfrak{so}_3$, where  $\mathfrak{so}_3$ consists of real skew-symmetric matrices $3\times 3$. Some of part (1) of Proposition \ref{SmithStafford1992Theorem0.3} for generic values of $\{\alpha,\beta, \gamma\}$ has been proved by Feigin and Odesskii \cite{FeiginOdesskii1989, FeiginOdesskii1989(2)}. 

The non-degenerate Sklyanin algebras or Sklyanin algebras with values of  $\{\alpha, \beta,\gamma\}$ equal to $\{-1,1,\gamma \}$, $\{\alpha,-1,1 \}$ or $\{1,\beta, -1 \}$ may be parameterized by pairs $(E,\tau)$ consisting of an elliptic curve $E$ and a translation automorphism $\tau: E\to E$. Details about the geometric data of this algebras can be found in Smith and Stafford \cite{SmithStafford1992}.

The following two results establish necessary and sufficient conditions on the parameters $\alpha,\beta$ and $\gamma$ for these algebras to be isomorphic. Note that it is not necessary that $\alpha,\beta$ and $\gamma$ satisfy the equality $\alpha+\beta+\gamma+\alpha\beta\gamma=0$.

\begin{proposition}[{\cite[Lemma 2.3]{ChirvasituSmith2023}}] \label{Isomorfimos-4Dim}
There are algebra isomorphisms 
\begin{align*}
    S(\alpha, \beta, \gamma) &\cong S(\beta, \gamma, \alpha)\cong S(\gamma, \alpha, \beta)\\
    &\cong S(-\alpha,- \gamma,-\beta)\cong S(-\beta,-\alpha, -\gamma)\cong S(-\gamma,-\beta,-\alpha).
\end{align*}
\end{proposition}

\begin{proposition}[{\cite[Proposition 2.4]{ChirvasituSmith2023}}] Suppose $\alpha\beta\gamma\neq 0$ and $\alpha'\beta'\gamma'\neq 0$. Then $S(\alpha,\beta,\gamma)\cong S(\alpha',\beta',\gamma')$ as graded $\Bbbk$-algebras if and only if $(\alpha',\beta',\gamma')$ is a cyclic permutation of either $(\alpha,\beta,\gamma)$ or $(-\alpha,-\gamma,-\beta)$.
\end{proposition}

\section{Differential and integral calculus}\label{DifferentialandintegralcalculusSkly}

This section contains the important results of the paper.
\begin{theorem}\label{DS3Skly} 
The non-degenerate three-dimensional Sklyanin algebras $S$ are differentially smooth.
\end{theorem}

\begin{proof}
Consider the following automorphisms:
\begin{align}
   \nu_{x}(x) = &\ -x, & \nu_{x}(y) = &\ -pq^{-1}y, & \nu_{x}(z) = &\ -p^{-1}qz, \label{Auto1} \\ 
    \nu_{y}(x) = &\ -p^{-1}qx, & \nu_{y}(y) = &\ -y, &  \nu_{y}(z) = &\ -pq^{-1}z, \label{Auto2} \\
    \nu_{z}(x) = &\ -pq^{-1}x, & \nu_z(y) = &\ -p^{-1}qy, & \nu_{z}(z) = &\ -z. \label{Auto3}
\end{align}
The maps $\nu_{x}$, $\nu_{y}$, and $\nu_{z}$ define an algebra homomorphism on $S$ precisely when their action on the generators $x$, $y$, and $z$ is compatible with the defining relations {\rm (}\ref{rel3Skly}{\rm )}. A straightforward computation confirms that this condition always holds. Moreover, since each map scales the generators by scalar multiples, the corresponding automorphisms commute pairwise.

Let $\Omega^{1}S$ be the free right $S$-module of rank three generated by the symbols $dx$, $dy$, and $dz$. We define a left $S$-module structure on $\Omega^{1}S$ by specifying, for all $p \in S$,
\begin{align}
    pdx = &\ dx \nu_{x}(p), \notag \\ \quad pdy = &\ dy\nu_{y}(p),\ {\rm and} \notag \\
    pdz = &\ dz\nu_{z}(p) \label{relrightmod}.
\end{align}

The relations in $\Omega^{1}S$ are given by 
\begin{align}
xdx = &\ -dx x, \notag \\
xdy = &\ -dyp^{-1}qx, \notag \\
xdz = &\ -dzpq^{-1}x, \label{rel1} \\
ydx = &\ -dxpq^{-1}y, \notag \\  
ydy = &\ -dyy, \notag \\
ydz = &\ -dzp^{-1}qy, \label{rel2} \\
zdx = &\ -dxp^{-1}qz, \notag \\
zdy = &\ -dypq^{-1}z, \ {\rm and} \notag \\
zdz = &\ -dzz. \label{rel3} 
\end{align}
    
We want to extend the correspondences 
\begin{equation*}
x \mapsto d x, \quad y \mapsto d y \quad {\rm and} \quad z\mapsto d z
\end{equation*} 

to a map $d: S \to \Omega^{1} S$ satisfying the Leibniz's rule. This is possible if it is compatible with the nontrivial relations {\rm (}\ref{rel3Skly}{\rm )}, i.e. if the equalities
\begin{align*}
        pdyz+pydz + qdzy+qzdy + rdx+rxdx  &\ = 0, \\
        pdzx+pzdx + qdxz +qxdz+ rdyy+rydy &\ = 0,\ {\rm and} \\
        pdxy+pxdy + qdyx+qydx + rdzz+rzdz &\ = 0
\end{align*}
hold.

Define $\Bbbk$-linear maps 
\begin{equation*}
\partial_{x}, \partial_{y}, \partial_{z}: S \rightarrow S
\end{equation*}

such that
\begin{align*}
    d(a)=dx\partial_{x}(a)+dy\partial_{y}(a)+dz\partial_{z}(a), \quad {\rm for\ all} \ a \in S.
\end{align*}

Since $dx$, $dy$, and $dz$ form a free generating set of the right $S$-module $\Omega^1S$, the above definitions yield well-defined left $S$-module actions. It is worth noting that $d(a) = 0$ holds if and only if $\partial_x(a) = \partial_y(a) = \partial_z(a) = 0$. Employing the relations in {\rm (}\ref{relrightmod}{\rm )} along with the definitions of the automorphisms $\nu_x$, $\nu_y$, and $\nu_z$, one obtains the following expressions:
\begin{align*}
\partial_{x}(x^k y^l z^s) &= (-1)^k\,k\,x^{k-1} y^l z^s,  \\
\partial_{y}(x^k y^l z^s) &= (-1)^{k+l}\,l\,p^{-k} x^k y^{l-1} z^s,  \\
\partial_{z}(x^k y^l z^s) &= (-1)^{k+l+s}\,s\,p^{k-l} q^{-k} x^k y^l z^{s-1}.
\end{align*}

Consequently, $d(a) = 0$ if and only if $a$ is a scalar multiple of the identity element. This establishes that the differential graded algebra $(\Omega S, d)$ is connected, where $\Omega S$ decomposes as $\Omega^0 S \oplus \Omega^1 S \oplus \Omega^2 S$.

Extending the differential $d$ to higher-degree forms in a manner consistent with relations {\rm (}\ref{rel1}{\rm )}, {\rm (}\ref{rel2}{\rm )}, and {\rm (}\ref{rel3}{\rm )}, we derive the following commutation rules in $\Omega^2 S$:
\begin{align}
dy \wedge dx &= pq^{-1}\,dx \wedge dy, \\
dz \wedge dx &= p^{-1}q\,dx \wedge dz, \quad \text{and} \\
dz \wedge dy &= pq^{-1}\,dy \wedge dz. \label{relsecondD1}
\end{align}

Given that the automorphisms $\nu_x$, $\nu_y$, and $\nu_z$ mutually commute, no additional constraints arise in $\Omega^2 S$ beyond those already stated. As a result, the space of 2-forms admits the decomposition
\begin{align*}
\Omega^2 S = dx \wedge dy\, S \oplus dx \wedge dz\, S \oplus dy \wedge dz\, S.
\end{align*}

Since $\Omega^3 S = \omega S \cong S$ both as a right and left $S$-module, where $\omega := dx \wedge dy \wedge dz$ and the associated automorphism is given by $\nu_{\omega} = \nu_{x} \circ \nu_{y} \circ \nu_{z}$, it follows that $\omega$ constitutes a volume form on $S$.
 From Proposition \ref{BrzezinskiSitarz2017Lemmas2.6and2.7} (2) we get that $\omega$ is an integral form by setting
\begin{align*}
\omega_1^1  = &\ \bar{\omega}_1^1 = dx, \\  
\omega_2^1 = &\ \bar{\omega}_2^1 = dy, \\
\omega_3^1 = &\ \bar{\omega}_3^1 = dz, \\
    \omega_1^2 = &\ dy\wedge dz, \\
    \omega_2^2 = &\ p^{-1}qdx\wedge dz, \\ 
    \omega_3^2 = &\ dx_1\wedge dx_3, \\
    \bar{\omega}_1^2 = &\ dy\wedge dz, \\
    \bar{\omega}_2^2 = &\ p^{-1}qdx\wedge dz, \quad {\rm and} \\ 
    \bar{\omega}_3^2 = &\ dx\wedge dy.
\end{align*}

By Proposition \ref{BrzezinskiSitarz2017Lemmas2.6and2.7} (2), we consider the expression $\omega' := dxa + dyb + dzc$ with $a, b, c \in \Bbbk$ and obtain the equality
\begin{align*}
    \sum_{i=1}^{3}\omega_{i}^{1}\pi_{\omega}(\bar{\omega}_i^{2}\wedge \omega') = &\ dx\pi_{\omega}(ady\wedge dz\wedge dx) \\
    & + dy\pi_{\omega}(p^{-1}qbdx\wedge dz\wedge dy) \\
    & + dz\pi_{\omega}(cdx\wedge dy\wedge dz) \\
    = &\ dxa+dyb+dzc \\
    = &\ \omega'.
\end{align*}

On the other hand, if $\omega'' := dx\wedge dya+dx\wedge dz b+dy\wedge dz c$ where $a,b,c \in \Bbbk$, we get that
\begin{align*}
\sum_{i=1}^{3}\omega_{i}^{2}\pi_{\omega}(\bar{\omega}_i^{1}\wedge \omega'') = &\  dy\wedge dz\pi_{\omega}(cdx\wedge dy \wedge dz) \\
&\ +pq^{-1}dx\wedge dz\pi_{\omega}(bdy\wedge dx\wedge dz) \\
    &\ + dx\wedge dy\pi_{\omega}(adz \wedge dx \wedge dy) \\ 
    = &\ dx\wedge dya+dx\wedge dzb+dy\wedge dz \\
    = &\ \omega''.
\end{align*}

As previously established, each graded component of the differential calculus is generated by forms of the type $\omega_i^j$ and $\bar{\omega}_i^{3-j}$, for $j = 1, 2$ and $i = 1, 2, 3$. Therefore, Proposition \ref{BrzezinskiSitarz2017Lemmas2.6and2.7} (2) ensures that $\omega$ serves as an integral form. In view of Proposition \ref{integrableequiva}, this implies that $(\Omega S, d)$ defines an integrable differential calculus of dimension $3$. Furthermore, since the Gelfand–Kirillov dimension of $S$ satisfies ${\rm GKdim}(S) = 3$ whenever $S$ is non-degenerate, it follows that $S$ is differentially smooth.
\end{proof}

\begin{remark}
In contrast with the non-degenerate case, degenerate Sklyanin algebras do not qualify as candidates for differential smoothness. As shown by Walton \cite[Theorem I.6 and Corollary III.9]{WaltonThesis2011}, these algebras exhibit exponential growth and consequently have infinite Gelfand–Kirillov dimension. Since differential smoothness requires the existence of a finite-dimensional integrable differential calculus whose dimension matches the Gelfand–Kirillov dimension, the infinite growth of degenerate Sklyanin algebras rules out the possibility of satisfying this property.
\end{remark}

\begin{theorem}\label{NODS4Skly} 
The four-dimensional Sklyanin algebras $S$ are not differentially smooth.
\end{theorem}

\begin{proof}
Suppose that $S$ has a first order differential calculus $\Omega S$ with $d: S \rightarrow \Omega^1S$ a derivation. By applying $d$ to $x_0 x_1  - x_1 x_0 =   \alpha(x_2x_3 + x_3x_2)$ we get that 
\[
d(x_0 x_1  - x_1 x_0) =  d(\alpha(x_2x_3 + x_3x_2)).
\]

Since $d$ is $\Bbbk$-linear, using the Leibniz's rule it follows that
\[
dx_0 x_1+x_0dx_1  - dx_1 x_0-x_1dx_0 =  \alpha dx_2x_3 +\alpha x_2dx_3 + \alpha dx_3x_2 + \alpha x_3dx_2.
\]

Using Property \ref{BrzezinskiSitarz20172.2}, the action of the module is written using automorphisms $\nu_{x_0}$, $\nu_{x_1}$ , $\nu_{x_2}$ and $\nu_{x_3}$
{\small{\[
dx_0 x_1+dx_1\nu_{x_1}(x_0)  - dx_1 x_0-dx_0\nu_{x_0}(x_1) =  \alpha dx_2x_3 +\alpha dx_3\nu_{x_3}(x_2) + \alpha dx_3x_2 + \alpha dx_2\nu_{x_2}(x_3).
\]}}

The terms that multiply the different differentials are put together to obtain
\[
dx_0(x_1-\nu_{x_0}(x_1))+dx_1(\nu_{x_1}(x_0) - x_0) -dx_2(\alpha x_3+\alpha\nu_{x_2}(x_3))-dx_3(\alpha\nu_{x_3}(x_2)+\alpha x_2) = 0.
\]
 
Since the differentials are generating elements for $\Omega^1S$ and all of them are equal to zero, the elements that multiply them are also zero. We obtain, in particular for $dx_0$, that 
\begin{align*}
    x_1-\nu_{x_0}(x_1) \Rightarrow \nu_{x_0}(x_1)=x_1.
\end{align*}
This implies that the automorphism $\nu_{x_0}$, corresponding to the differential calculus behavior, must leave the generator $x_1$ invariant.

Now, considering the relation $x_0x_1 + x_1x_0 = x_2x_3 - x_3x_2$, in a manner entirely analogous to the preceding argument, it follows that
\[
dx_0(x_1+\nu_{x_0}(x_1))+dx_1(\nu_{x_1}(x_0) + x_0) -dx_2( x_3+\nu_{x_2}(x_3))+dx_3(\nu_{x_3}(x_2)+ x_2) = 0.
\]
Again, focusing on the coefficient of $dx_0$, it must vanish; that is
\begin{align*}
    x_1+\nu_{x_0}(x_1) \Rightarrow \nu_{x_0}(x_1)=-x_1.
\end{align*}
It follows that the automorphism $\nu_{x_0}$, when acting on $x_1$, would yield two distinct images, which is a contradiction. Thus, no such first-order differential calculus $\Omega^1 S$ exists, and consequently, $S$ is not differentially smooth.
\end{proof}

It is important to emphasize that the previous theorem applies uniformly to both degenerate and non-degenerate four-dimensional Sklyanin algebras. Indeed, the proof does not rely on any assumption regarding the Gelfand–Kirillov dimension of the algebra, but rather on structural obstructions arising from the failure to construct a compatible differential calculus satisfying the integrability and connectedness conditions. Therefore, regardless of whether the algebra has finite or infinite Gelfand–Kirillov dimension, none of the four-generator Sklyanin algebras are differentially smooth.

\section{Future Work}

The results presented in this paper suggest several promising directions for future investigation. One natural extension would be to consider the differential smoothness of other families of Artin–Schelter regular algebras, particularly those arising from quantum homogeneous spaces or from generalized Sklyanin-type constructions in higher dimensions. It would be especially interesting to explore whether the techniques used here (namely, the explicit construction of differential calculi and analysis of volume forms) can be adapted to settings involving twisted Calabi–Yau algebras or more general PBW deformations.

Another compelling direction is the interaction between homological smoothness and differential smoothness in the context of noncommutative projective geometry. Given that some non-differentially smooth Sklyanin algebras are nonetheless homologically smooth, understanding the precise algebraic and geometric mechanisms behind this divergence could shed light on deeper categorical and deformation-theoretic properties. Furthermore, it would be worthwhile to investigate how the absence of differential smoothness influences representation theory, particularly the existence of point modules and their moduli spaces.

\section{Declarations}

The three authors were supported by Dirección Ciencias Básicas, Universidad ECCI- sede Bogotá.

All authors declare that they have no conflicts of interest.

The authors did not receive any funding while working on this paper.

\end{document}